\documentclass{birkjour}

\usepackage{amsmath,amsthm,amsfonts,amssymb,graphicx}
\usepackage{booktabs,enumerate,enumitem}
\providecommand{\norm}[1]{\lVert#1\rVert}
\providecommand{\abs}[1]{\lvert#1\rvert}
\providecommand{\N}{\mathbb{N}}
\providecommand{\R}{\mathbb{R}}
\newcommand{\p}[1]{\langle #1 \rangle}

\DeclareMathOperator{\co}{conv}
\DeclareMathOperator{\inte}{int}

  \usepackage{mathabx}
\newcommand{\g}[1]{\ldbrack #1 \rdbrack}

\theoremstyle{theorem}
\newtheorem{theorem}{Theorem}[section]

\newtheorem{corollary}{Corollary}[section]
\theoremstyle{definition}
\newtheorem{definition}{Definition}[section]
\newtheorem{example}{Example}[section]

\theoremstyle{remark}
\newtheorem{remark}{Remark}[section]

\begin{document}
	
	%
	%
	%
	%
	%
	%
	%
	%
	%

\title[The Bolzano-Poincar\'e-Miranda Theorem]{The Bolzano-Poincar\'e-Miranda theorem in infinite dimensional Banach spaces}

\author[Ariza-Ruiz]{David Ariza-Ruiz}
\address[label1]{Departamento de An\'alisis Matem\'atico,
	Universidad de Sevilla, Apdo. 1160, 41080-Sevilla, Spain}
\email{dariza@us.es }
\author[Garcia-Falset]{Jes\'{u}s Garcia-Falset}
\address[label2]{Departament d'An\`{a}lisi Matem\`{a}tica,
	Universitat de Val\`{e}ncia, Dr. Moliner 50, 46100, Burjassot,
	Val\`{e}ncia, Spain}
\email{garciaf@uv.es}
\author[Reich]{Simeon Reich}
\address[label3]{Department of Mathematics,
	The Technion -- Israel Institute of Technology, 32000 Haifa, Israel}
\email{sreich@technion.ac.il (the corresponding author)}


\begin{abstract}
We study the existence of zeroes of mappings defined in Banach spaces.
We obtain, in particular, an extension of the well-known
Bolzano-Poincar\'e-Miranda theorem to infinite dimensional Banach spaces.
We also establish a result regarding the existence of periodic solutions
to differential equations posed in an arbitrary Banach space.
\end{abstract}

\keywords{
Bolzano's theorem, the Bolzano-Poincar\'e-Miranda theorem, zeroes of an
operator.}

\subjclass{Primary 47H10; Secondary 47J05}

\maketitle
%
\section{Introduction}\label{sec:intro}

The first general result regarding the existence of a zero of a function
is Bolzano's theorem (also called the intermediate value theorem), which
can be stated as follows:

\begin{theorem}[Bolzano's theorem]
If $f:[a,b]\to\R$ is a continuous function and $f(a)f(b)<0$, then there
exists a point $x\in (a,b)$ such that $f(x)=0$.
\end{theorem}
The first proofs of this theorem were given independently by
Bolzano~\cite{B} in 1817 and by Cauchy \cite{C} in 1821.
The following theorem is one of the first generalizations of Bolzano's
theorem.
\begin{theorem}\label{Mi}
Let $P := \{x\in\mathbb{R}^n: \ |x_i|\leq L \ {\rm for \ all} \ 1\leq
i\leq
n\}$. Suppose that $F=(f_1,\cdots,f_n):P\to\mathbb{R}^n$ is a
continuous mapping on $P$ such that
\begin{enumerate}
\item[(a)] $f_i(x_1,\cdots,x_{i-1},-L,x_{i+1},\cdots,x_n)>0$
\; \; for \; \; $1\leq i\leq n$,
\item[(b)] $f_i(x_1,\cdots,x_{i-1},L,x_{i+1},\cdots,x_n)<0$
\; \; for \; \; $1\leq i\leq n$.
\end{enumerate}
Then there exists at least one point $x\in P$ such that $F(x)=0.$
\end{theorem}

As recalled in the nice survey paper~\cite{Ma}, Theorem~\ref{Mi}
was stated in 1883 by Poincar\'e in~\cite{P} without proof and
then forgotten for a long time. It was rediscovered by Miranda, who in
1940 showed that it was equivalent to Brouwer's fixed point theorem
(see~\cite{M,V}). Since then, this result has been called
the Bolzano-Poincar\'e-Miranda theorem. Poincar\'e was mainly motivated by
the study of periodic solutions to differential equations arising from the
three-body problem, which can be obtained as fixed points of what is now
called the Poincar\'e map. After his work, many other authors found
applications of Theorem~\ref{Mi} and different generalizations have been
established (see \cite{AFHM,FG,K,V,Vr,S} and references therein).
Theorem~\ref{Mi} also turns up in a recent detailed study of calibration~\cite{H}. 
Since most of the generalizations of the Bolzano-Poincar\'e-Miranda
theorem have been obtained in finite dimensional spaces, the following
question arises naturally: can Theorem~\ref{Mi} be extended to infinite
dimensional spaces? Several partial positive answers to this question
have been given; see, for instance, \cite{A,I,Mo,Sc}. 
In this connection, see also~\cite{RT}, which concerns certain operators 
which are not necessarily continuous. 
In the present paper we answer, {\it inter alia}, this question in the affirmative and obtain a generalization of the Bolzano-Poincar\'e-Miranda theorem
which holds in {\it any} Banach space.

\section{Notations and Preliminaries}

Throughout this paper $(X,\norm{\cdot})$ is a Banach space and $X^*$
is its topological dual. As usual, $B_r[x]$ and $S_r(x)$ denote the
closed ball and the closed sphere of center $x$ and radius $r>0$,
respectively. For a nonempty subset $C$ of $X$, we denote by $\overline{C}$, $\partial C$, $\inte(C)$ and $\co(C)$ the closure, the boundary, the interior and the convex hull of $C$, respectively. For $x\in X$, we denote by $J(x)$
the value of the normalized duality mapping $J$ at $x$, to wit
$J(x):=\{j\in X^*: \ \langle x,j\rangle:=j(x)=\|x\|^2, \ \|j\|=\|x\|\}$.
In general, the normalized duality mapping $J:X\to 2^{X^{*}}$ is
a set-valued mapping. For more details, see~\cite{Cioranescu} and its
review~\cite{Reich}.

\begin{definition}
A Banach space $X$ is said to be {\it smooth} if the normalized duality
mapping $J:X\to 2^X$ is single-valued.
\end{definition}


The following classical fixed point theorems can be found, for instance,
in \cite{De,Sm,Z}. It is worth noting that in each reference we find
different proofs.

\begin{theorem}[Brouwer's fixed point theorem] Let $D \subseteq \mathbb{R}^n$ be a nonempty
bounded, closed and convex subset of $\mathbb{R}^n$. If $f:D\to D$ is
continuous, then it has a fixed point in $D$.
\end{theorem}

Let $( X,\norm{\cdot})$ be a real Banach space and let $K$ a nonempty
subset of $X$. A mapping $f:K\to X$ is said to be {\it compact} if $f(C)$
is relatively compact whenever $C$ is a bounded subset of $K$.
If, moreover, $f$ is continuous on $K$, then $f$ is called
{\it completely continuous}. It is clear that every continuous mapping
with a closed domain in a finite dimensional Banach space is, in fact,
completely continuous.

\begin{theorem}[Schauder's fixed point theorem] Let $C$ be a nonempty
closed and convex subset of a Banach space $X$. If $f:C\to C$ is
completely continuous, then it has a fixed point.
\end{theorem}

\section{The main result}

Our goal in this section is to study the existence of zeroes for an
operator $f$ defined in a general Banach space $X$. To this end,
we introduce a class of functionals $\g{\cdot,\cdot}:X\times X\to\R$
satisfying the following two conditions:
\begin{enumerate}[label=($C_\arabic*$)]
\item $\g{x,x}>0$ for all $x\in X$ with $x \neq 0$;
\item $\g{\lambda\,x,x}=\lambda\,\g{x,x}$ for all $x\in X$ and $\lambda\in\R$.
\end{enumerate}

\begin{example}
\begin{enumerate}[label=(\arabic*)]
\item If $X$ is a pre-Hilbert space and $\p{\cdot,\cdot}$ is its inner
product, then we can set $\g{\cdot,\cdot}=\p{\cdot,\cdot}$.
\item In general, if $X$ is a Banach space, we may define,
using the duality mapping, a functional $\g{\cdot,\cdot}:X\times X\to\R$
either by
$$\g{x,y}=\p{x,y}_{+}:=\sup_{j(y)\in J(y)}\p{x,j(y)}=\max_{j(y)\in J(y)}\p{x,j(y)}$$
or by
$$\g{x,y}=\p{x,y}_{-}:=\inf_{j(y)\in J(y)}\p{x,j(y)}=\min_{j(y)\in J(y)}\p{x,j(y)}.$$
At this point the interested reader is referred to~\cite{De}, where some
essential properties of
$J$ and $\p{\cdot,\cdot}_{\pm}$ are collected in Propositions~12.3 and~13.1.
In particular, if $X$ is a smooth Banach space,
then $\g{\cdot,\cdot}$ can be given by $\g{x,y}=\p{x,J(y)}$.
In this connection, see also \cite{Cioranescu} and its review
\cite{Reich}.

\end{enumerate}
\end{example}

\begin{theorem}\label{tgb}
Let $U$ be a nonempty, bounded and closed subset of a Banach space $X$
with ${\rm int}(U)\neq\emptyset$ and let $f:U\to X$ be a completely continuous mapping. If there exist a functional ${\small\g{\cdot,\cdot}}:X\times X\to\R$
satisfying $(C_1)$ and $(C_2)$, and a point $z\in\inte(U)$ such that $\g{f(x),x-z}$ has a constant sign for
all $x\in\partial U$, then $0\in\overline{f(U)}$. Moreover, if  $0\notin\partial f(U)\setminus f(U)$, then  $0\in f(U)$.
\end{theorem}

\begin{proof}
Without any loss of generality we may assume that $\g{f(x),x-z}\geq0$ for
every $x\in \partial U$. This is because otherwise this expression would
be non-positive and we could use an analogous   reasoning for $-f$.
\par\smallskip
Fix $n\in\N$ and consider the mapping $g_n:U\to X$ defined by
$g_n(x) := -nf(x)+z$. It is clear that $g_n$ is a completely continuous
mapping. Hence there exists a closed ball  $\Omega_n$ such that $U\cup
g_n(U)\subseteq \Omega_n.$
\par\smallskip
Now we may define $D_n:=\big\{x\in U: \ x=\lambda g_n(x)+(1-\lambda)z \
{\rm for \  some} \ \lambda\in [0,1]\big\}$.
We claim that $D_n$ is a nonempty compact subset of $X$.
Indeed, taking $\lambda=0$, we see that $z\in D_n$ and so $D_n$ is
nonempty. We also have $D_n\subseteq \overline{\co}\big(\{z\}\cup
g_n(U)\big)$,
which implies that the set $D_n$ is compact because the mapping $g_n$
is completely continuous.
\par\smallskip
We claim that $D_n\cap(\Omega_n\setminus\,\inte(U))=\emptyset.$
\par\smallskip
Suppose to the contrary that $x\in D_n\cap(\Omega_n\setminus \inte(U))$.
Then $x \in \partial U$ and
by hypothesis, $\g{f(x),x-z}\geq 0$. However, since $x\in D_n$,
there is $\lambda\in (0,1]$ such that $x=\lambda g_n(x)+(1-\lambda)z,$
that is, $x-z=\lambda(g_n(x)-z)=-\lambda nf(x)$. This implies that
$$\g{f(x),x-z}=\g{-\tfrac{1}{n\,\lambda}(x-z),x-z}
=-\frac{1}{n\,\lambda}\g{x-z,x-z}<0,$$
which is a contradiction.
\par\smallskip
Now by Urysohn's lemma (see, for instance, \cite[page 146]{J}),
there exists a continuous function $\varphi:\Omega_n\to [0,1]$
such that $\varphi(x)=1$ whenever $x\in D_n$ and $\varphi(x)=0$
for all $x\in\Omega_n\setminus {\rm int}(U)$.
This allows us to define
$F_n:\Omega_n\to\Omega_n$ by
$$F_n(x):= \varphi(x)g_n(x)+(1-\varphi(x))z.$$
Therefore, since $F_n$ is a continuous and compact self-mapping of a
closed, convex and bounded subset, Schauder's theorem yields the
existence of a point $x_n\in\Omega_n$ such that $F_n(x_n)=x_n$.
\par\smallskip
If $x_n\notin {\rm int}(U)$, then $\varphi(x_n)=0$ and then $x_n=z$,
which contradicts the hypothesis $z\in\inte(U)$.
Thus $x_n\in \inte(U)$ and $x_n=\varphi(x_n)g_n(x_n)+(1-\varphi(x_n))z$,
which implies that $x_n\in D_n$, and so $\varphi(x_n)=1.$ Therefore
$x_n=-nf(x_n)+z$, that is,
$$f(x_n)=\frac{1}{n}(z-x_n).$$
Since $x_n\in U$ and $U$ is bounded, we conclude that
$f(x_n)\to 0$ as $n\to\infty$ and so $0\in\overline{f(U)}$.
\par\smallskip
Finally, if  $0\notin\partial f(U)\setminus f(U)$, then $0 \in f(U)$
because $0\in\overline{f(U)}$.
\end{proof}

\begin{remark}
If, in addition to the assumptions of Theorem~\ref{tgb}, we assume that
$X$ is a reflexive Banach space, $U$ is a nonempty, closed and convex
subset of $X$ with $\inte(U)\neq \emptyset$,
and $f:U\to X$ is weak-strong continuous on $U$, then by Mazur's theorem,
$f(U)$ is a closed set and therefore $0\in f(U)$.
\end{remark}

\begin{corollary}\label{gb}
Given a Banach space $X$, a point $z \in X$ and a number $r>0$,
let $f:B_r[z]\to X$ be a completely continuous mapping. If there exists a 
functional ${\small\g{\cdot,\cdot}}:X\times X\to\R$ satisfying $(C_1)$ and $(C_2)$ such that
$\g{f(x),x-z}$ has a constant sign for all $x\in S_r(z)$, then
$0\in\overline{f(B_r[z])}$. Moreover, if $X$ is reflexive
and $f$ is weak-strong continuous on $B_r[z]$, then  $0\in f(B_r[z])$.
\end{corollary}

\begin{proof}
It is enough to consider the closed ball
$U := \big\{x\in X: \,\norm{x-z}\leq r\big\}$ and then apply
Theorem~\ref{tgb}.
\end{proof}

\begin{corollary}\label{Pb}
Let $U$ be a nonempty, bounded and closed subset of a Banach space $X$
with $\inte(U)\neq\emptyset$. If $f:U\to X$ is a completely continuous
mapping and there exists a point $z \in \inte(U)$ such that either
$f(x)\notin\{\lambda(x-z): \ \lambda<0\}$ for all $x\in\partial U$ or $f(x)\notin\{\lambda(x-z): \ \lambda>0\}$ for all $x\in\partial U$ , then $0\in\overline{f(U)}$.
\end{corollary}

\begin{proof}
In order to obtain this result, it is enough to define the
following functional:

$$\g{x,y} := \left\{\begin{array}{lll}\langle x,y\rangle_+
& {\rm if} x=\lambda y \ {\rm for \ some} \ \lambda\in\mathbb{R}
\\0& {\rm otherwise}.\end{array}\right.$$
It is clear that $\g{\cdot,\cdot}$ satisfies conditions $(C_1)$ and $(C_2)$.

If $f(x)\notin\{\lambda(x-z): \ \lambda<0\}$ for all $x\in\partial U$,
then $\g{f(x),x-z}\geq 0$ for all $x\in \partial U$.
If, on the other hand, $f(x)\notin\{\lambda(x-z): \ \lambda>0\}$ for all
$x\in\partial U$, then we obtain $\g{f(x),x-z}\leq 0.$ In both cases,
the sign of $\g{f(x),x-z}$ is constant. Therefore it follows from
Theorem \ref{tgb} that $0\in\overline{f(U)}$, as asserted.
\end{proof}

\begin{remark}
As a consequence of the previous result, we obtain
Proposition~4 in~\cite{AFHM}, Theorem~1 in~\cite{Mo},
Corollary~3 in~\cite{I} and the Poincar\'e-Bohl theorem~\cite[Theorem
2]{FG}.
\end{remark}

\subsection{Consequences in finite dimensional spaces}

In~\cite[Theorem 3.2]{AFHM} the authors present a generalization
of the Bolzano-Poincar\'e-Miranda theorem which contains
Kantorovich's result~\cite[Theorem~2.1]{AFHM}.
Now we improve upon this result: applying Corollary~\ref{Pb},
we show that it also holds for convex bodies.

Recall that a subset $D$ of $\mathbb{R}^n$ is said to be a
{\it convex body} if it is a compact and convex set with nonempty interior
($D=\overline{\inte(D)}$). Given a point $x\in \partial D$,
we define the {\it  normal cone} to $\inte(D)$ at $x$ by
$\mathcal{N}_{D}(x):=\{\,v\in\R^n: \ \p{v,y-x}< 0,
\ {\rm for \ every} \ y\in\inte(D)\}$, where, as usual,
$\p{\cdot,\cdot}$ denotes the Euclidean scalar product in $\mathbb{R}^n.$
\par\medskip

\begin{corollary}\label{cb}
Let $D$ be a convex body in $\mathbb{R}^n$ endowed with
an arbitrary norm $\|\cdot\|$. Let $f:D\to\mathbb{R}^n$ be a continuous
mapping and assume that for all $x\in \partial D$, there exists a point
$a(x)\in\mathcal{N}_{D}(x)$ such that $\p{f(x),a(x)}\geq 0$. Then $f$ has
a zero in $D.$
\end{corollary}

\begin{proof}
Let $x_0$ be an element in ${\inte(D)}$. By Corollary \ref{Pb}, it is
enough to show that $f(x)\notin\{\lambda(x-x_0): \ \lambda<0\}$ for all
$x\in \partial D.$ Suppose to the contrary that there exist a point $x\in \partial D$
and a number $\lambda<0$ such that $f(x)=\lambda(x-x_0)$. In this case we have
$0\leq\p{a(x),f(x)}=\lambda\p{a(x),x-x_0}$, which means that
$\p{a(x),x_0-x}\geq 0$ and therefore $a(x)\notin \mathcal{N}_{D}(x)$, which is a contradiction.
\end{proof}

Next we give proofs of Bolzano's theorem and the Bolzano-Poincar\'{e}-Miranda theorem
by employing our results.

\begin{proof}[Proof of Bolzano's theorem]
Let $X=\mathbb{R}$ and $\|\cdot\|=|\cdot|$. It is clear that
$(X,\|\cdot\|)$ is a smooth reflexive Banach space. Taking
$z=\frac{a+b}{2}$ and $r=\frac{b-a}{2}$ we have $B_r[z]=[a,b]$
and $S_r(z)=\{a,b\}$. We define $\g{x,y} := x\,y$. The hypothesis
$f(a)f(b)<0$ implies that  $\g{f(x),x-z}=f(x)(x-z)$ has a constant sign
for all $x\in S_r(z)$. Finally, since  $(X,\|\cdot\|)$ is a finite
dimensional Banach space, we may apply Corollary~\ref{gb} abd obtain the result.
\end{proof}

\begin{proof}[Proof of the Bolzano-Poincar\'{e}-Miranda theorem]
It is clear that $P$ is a convex body in $\mathbb{R}^n$ and that
$\partial P=\{x\in P: \ x_i=\pm L \ {\rm for \ some} \ 1\leq i\leq n\}$.
Moreover, if $x\in \partial P$, then
$x=(x_1,\cdots,x_{i-1},\pm L,x_{i+1},\cdots,x_n)$. In this case, we take
$a(x)=(0,\ldots,0,\pm L,0,\ldots,0)\in \mathcal{N}_{P}(x)$. Now we can
work with the function $g:=-f.$ Bearing in mind conditions~(\emph{a}) and~(\emph{b}), we see that
$$\langle g(x),a(x)\rangle=-f_i(x_1,\cdots,x_{i-1},\pm L,x_{i+1},\cdots,x_n)(\pm L)>0.$$
The above argument shows that all the assumptions of Corollary~\ref{cb}
are satisfied and this allows us to conclude the proof.
\end{proof}

\subsection{Consequences in infinite dimensional spaces}

As we have already mentioned in the introduction,
in~\cite[Theorem 1]{A} Avramescu
gave a partial answer to the question
regarding the Bolzano-Poincar\'e-Miranda
theorem in infinite dimensional spaces. We now show that this result
can be obtained as a corollary of Theorem~\ref{tgb}.

To this end, let $\mathbb{B}$ be a Banach space endowed with an inner product
$\p{\cdot,\cdot}$ (which does not necessarily have any relation with the
norm $\norm{\cdot}$ of $\mathbb{B}$) and let $K$ be an arbitrary
compact topological space.
Set $X:=\{x:K\to \mathbb{B}:x \,\text{ is continuous}\}$, define on $X$
the usual norm $\|x\|_{\infty}=\max\{\|x(t)\|: \ t\in K\}$ and
let $B := \{x\in X: \ \|x\|_{\infty}\leq 1\}$.

\begin{corollary}(\cite[Theorem 1]{A})
Let $f:B\to X$ be a continuous and compact mapping satisfying the following assumptions:
\begin{enumerate}
\item[$\boldsymbol{(a)}$] $0\notin \partial f(B)\setminus f(B),$
\item[$\boldsymbol{(b)}$] $\p{f(x)(t),x(t)}\leq 0$ for all $x\in \partial B$ and for all $t\in K$.
\end{enumerate}
Then there exists at least one element $x\in B$ such that $f(x)=0$.
\end{corollary}

\begin{proof}
Given $y\in X$, since $K$ is compact, it is clear that there exists
$t_0\in K$ such that $\|y\|_{\infty}=\|y(t_0)\|.$
Let $t_y:=\min\big\{t\in K:\norm{y(t)}=\norm{y}_\infty\big\}$.
We define $\g{\cdot,\cdot}:X\times X\to\R$ by $\g{x,y}:=\p{x(t_y),y(t_y)}$.
It is not difficult to see that $\g{\cdot,\cdot}$ satisfies conditions
$(C_1)$ and $(C_2)$. Thus, if $x\in \partial B$, we have
$\g{f(x),x}=\p{f(x)(t_x),x(t_x)}\leq 0$ by using assumption (\emph{b}).
This means that $\g{f(x),x}$ has a constant sign for all $x\in \partial B$.
Now, since $f$ satisfies assumption (\emph{a}), we can apply
Theorem~\ref{tgb} to conclude that $0\in f(B)$, as asserted.
\end{proof}

In \cite{Sc} the author established a version of the 
Bolzano-Poincar\'e-Miranda theorem in the Hilbert space $\ell_2$. 
We are now going to see that such a result can be obtained as a consequence 
of Theorem \ref{tgb}. For the sake of simplicity, we prove the 
following version of the result. 

\begin{corollary}(\cite[Theorem 2.3]{K})
Let $C:=\{x\in\ell_2: \ |x_k|\leq k^{-1}\}$ be the Hilbert cube 
and let $f:C\to \ell_2$ be continuous. If for all $k\in\mathbb{N}$, we 
have
$$
f_k(x_1,\ldots,x_{k-1},-\tfrac{1}{k},x_{k+1},\ldots)\geq 0
\quad\mbox{and}\quad
f_k(x_1,\ldots,x_{k-1},\tfrac{1}{k},x_{k+1},\ldots)\leq 0,$$
then $f$ has a zero.
\end{corollary}

\begin{proof}
Set $l:=\|\{\frac{1}{k}\}_{k\in\N}\|_2$ and consider $B_{l}[0].$ 
It is clear that $C\subseteq B_{l}[0].$ Let $P_C$ denote the metric 
projection onto $C$. Given $x=\{x_k\}_{k\in\N}\in B_l[0]$, it is not 
difficult to see that
$$
P_C(x)_{k}=\left\{\begin{array}{cll}
x_k& \mbox{if} & |x_k|<\frac{1}{k} \\[.5em]
\frac{1}{k}& \mbox{if} & x_k\geq\frac{1}{k}\\[.5em]
-\frac{1}{k}& \mbox{if} & x_k\leq-\frac{1}{k}.
\end{array}\right.$$

Consider now the properties of the mapping $F:B_{l}[0]\to\ell_2$ 
defined by $F(x):=f(P_C(x))$. 
\par\smallskip
Since $\ell_2$ is a Hilbert space, $P_C$ is a continuous (even 
nonexpansive) mapping. Hence, 
since $C$ is a compact set and  $C\subseteq B_{l}[0]$, 
the mapping $F$ is completely continuous and $F(B_l[0])$ is also a 
compact set.
\par\smallskip
Given $y=\{y_k\}_{k\in\N}\in \ell_2$, we define
$$
i_y:=\left\{\begin{array}{lll}
\min\{k\in\mathbb{N}: \ |y_k|=\|y\|_{\infty}\} & \mbox{if} & |y_k|<k^{-1} \quad\forall\;k\in \mathbb{N} \\[.5em]
\min\{k\in\mathbb{N}: \ |y_k|\geq k^{-1}\} & \mbox{if} 
& \exists\,k\in\mathbb{N} \; \mbox{ with } \;  |y_k|\geq k^{-1}.
\end{array}\right.$$
Now we define $\g{x,y}:=x_{i_{y}}y_{i_{y}}$. 
It's clear that $\g{\cdot,\cdot}$ satisfies conditions $(C_1)$ and $(C_2).$ 
Moreover, if $x\in S_{l}(0)$, then  
$|x_{i_{x}}|\geq\frac{1}{i_{x}}$ and since
$\g{F(x),x}=f_{i_{x}}(P_C(x))x_{i_{x}}$, we see that $\g{F(x),x}\leq 0$ 
for all $x\in S_l(0)$. These conditions guarantee that we can apply 
Theorem~\ref{tgb}. Hence there exists a point $x_0\in B_l[0]$ such that 
$F(x_0)=0$. This, in its turn, implies that $f$ has a zero in $C$,
as claimed. 
\end{proof}

\section{Applications}

In this section we apply the results of the previous section in order to 
establish the existence of at least one solution to systems of nonlinear 
equations and also the existence of a periodic solution in a general 
Banach space to the differential equation $x'(t)=f(t,x(t))$, where 
$t\in[0,T]$.

\subsection{An application to systems of nonlinear equations}

\begin{theorem}\label{thm:appl:finite}
Let $L$ be a continuous linear mapping from $(\R^n,\norm{\cdot})$ into 
itself. Assume that its inverse $L^{-1}$ exists and that 
$\ell:=\min\big\{\norm{L(x)}:\norm{x}=1\big\}>0$. 
Let $g:\R^n\to\R^n$ be a continuous mapping. If there exists $R>0$ 
such that $\norm{g(x)}\leq \ell\,R$ for all $x\in B_R[0]$, 
then the nonlinear equation $L(x)+g(x)=0$ has at least one solution in~$B_R[0]$.
\end{theorem}

\begin{proof}
Note that $L(x)+g(x)=0$ is equivalent to $x+L^{-1}(g(x))=0$. 
Thus, defining $F:\R^n\to\R^n$ by $F(x):= x+L^{-1}(g(x))$, we see 
that it is enough to prove the existence of a zero of $F$ in $B_R[0]$.
\par\medskip
It is clear that $F$ is continuous and, therefore, compact. 
We claim that $\norm{L^{-1}(x)}\leq \frac{1}{\ell}\,\norm{x}$ for all 
$x\in\R^n$. To see this, observe that the inequality 
$$
\norm{x}=\norm{L(L^{-1}(x))}\geq \ell\,\norm{L^{-1}(x)}
$$
implies that $\norm{L^{-1}(x)}\leq \dfrac{1}{\ell}\,\norm{x}$.
\par\medskip
Furthermore, if we set $\g{x,y}=\p{x,y}_+$, then for each $x\in 
S_R[0]$, we have 
\begin{equation*}
\begin{split}
\g{F(x),x}
&\geq\p{x+L^{-1}(g(x)),j}=R^2+\p{L^{-1}(g(x)),j} \\
&\geq R^2-\norm{L^{-1}(g(x))}\,R \geq R^2-\frac{R}{\ell}\,\norm{g(x)}
\geq 0,
\end{split}
\end{equation*}
where $j \in J(x)$. Hence, by Corollary~\ref{gb}, there exists a 
point $x_0 \in B_R[0]$ such that $F(x_0)=0$, that is, $L(x_0)+g(x_0)=0$.
\end{proof}

\begin{corollary}\label{cor:appl:finite}
Let $L$ be a continuous linear mapping from $(\R^n,\norm{\cdot})$ into 
itself. Assume that its inverse $L^{-1}$ exists and that  
$\ell:=\min\big\{\norm{L(x)}:\norm{x}=1\big\}>0$. 
Let $g:\R^n\to\R^n$ be a continuous mapping. If there exist numbers  
$\alpha,\beta\geq 0$, with $\alpha<\ell$, such that 
$\norm{g(x)}\leq\alpha\,\norm{x}+\beta$ for all $x\in \R^n$, 
then the nonlinear equation $L(x)+g(x)=0$ has at least one solution 
in the ball $B_R[0]$, where $R=\beta/(\ell-\alpha)$.
\end{corollary}

\begin{proof}
Since $\ell-\alpha>0$ by assumption, we may indeed consider  
$R := \beta/(\ell-\alpha)$. To see that 
the equation $L(x)+g(x)=0$ has at least one solution in the ball $B_R[0]$, 
note that for each point $x\in B_{R_0}[0]$, we have
$$
\norm{g(x)}\leq \alpha\,\norm{x}+\beta\leq \alpha\,R+\beta\leq \alpha\,R_0+(\ell-\alpha)\,R=\ell\,R.
$$
Therefore the result follows from the preceding theorem.
\end{proof}

\begin{example}
The system of nonlinear equations
$$
(S)\left\{
\begin{array}{l}
-2x+7y+4y\,\cos(x+2y)=3 \\[.5em]
7x-2y+3x\,\sin(x-3y)=2
\end{array}
\right.
$$
has at least one solution in the square $[-3,3]\times[-3,3]$. To see 
this, it is enough to 
use the preceding result for the plane $\R^2$ endowed with the norm 
$\norm{(x,y)}_\infty=\max\{\abs{x},\abs{y}\}$, the continuous linear mapping 
$L(x,y)=(-2x+7y,7x-2y)$ and the continuous nonlinear mapping 
$g(x,y)=(4y\,\cos(x+2y)-3,3x\,\sin(x-3y)-2)$. It is not difficult to see 
that $\ell=5$ and $\norm{g(x,y)}_\infty\leq 4\norm{(x,y)}_\infty+3$ 
for all $(x,y)\in\R^2$. Thus, by Corollary~\ref{cor:appl:finite}, there 
indeed exists at least one solution of~$(S)$ in 
$B_3[0]=[-3,3]\times[-3,3]$. 
\end{example}

\subsection{An application to differential equations}

Let $(X,\norm{\cdot})$ be a Banach space and let $f:[0,T]\times X\to X$ be 
a continuous mapping which is Lipschitz with respect to the second variable, 
that is, there exists a constant $L>0$ such that 
$\norm{f(t,x)-f(t,y)}\leq L\,\norm{x-y}$ for all $t\in[0,T]$ and $x,y\in X$.
\par\medskip
Consider the differential equation
\begin{equation}\label{eq:dif.1}
x'(t)=f(t,x(t)), \qquad t\in[0,T].
\end{equation}
According to the Picard-Lindel\"{o}f theorem (see~\cite[Chapter 3]{Z}), 
given a point 
$a\in X$, there exists a unique continuous and differentiable function 
$x_a:[0,T]\to X$ satisfying \eqref{eq:dif.1} with $x(0)=a$. Moreover, this 
solution satisfies
$$x_a(t)=a+\int_0^t f(s,x_a(s))\,ds.$$
Now define the mapping $F:X\to X$ by $F(a) := -a+x_a(T)$.
\par\medskip
We claim that $F$ is continuous. Indeed, given $a,b\in X$, we have, for 
all $t\in[0,T]$, 
\begin{equation*}
\begin{split}
\norm{x_a(t)-x_b(t)}
&\leq\norm{a-b}+\int_0^t \norm{f(s,x_a(s))-f(s,x_b(s))}\,ds \\
&\leq\norm{a-b}+\int_0^t L\,\norm{x_a(s)-x_b(s)}\,ds.
\end{split}
\end{equation*}
By using Gronwall's inequality (see, for instance, \cite[Proposition 
3.10]{Z}), we obtain $\norm{x_a(t)-x_b(t)}\leq\norm{a-b}\,e^{L\,t}
$ for all $t\in[0,T]$. Therefore 
$$\norm{F(a)-F(b)}\leq\norm{a-b}+\norm{a-b}\,e^{L\,T}=(1+e^{L\,T})\,\norm{a-b}.$$
Thus $F$ is (even Lipschitz) continuous.
\par\medskip
Suppose that there exists a number $R>0$ such that for each point $x\in X$ 
with $\norm{x}=R$, there exists $j\in J(x)$ with $\p{f(t,x),j}\leq0$ 
for all $t\in[0,T]$.
\par\medskip
In this case, bearing in mind Kato's differentiation rule, we get
$$\norm{x(t)}\,\frac{d}{dt}\norm{x(t)}=\p{x'(t),j},$$
which implies that
$$\frac{d}{dt}\norm{x(t)}^2=2\p{f(t,x(t)),j}$$
for any $j\in J(x(t))$, whenever $x$ is a solution to the 
differential equation \eqref{eq:dif.1}.
\par\medskip
We claim that if $a\in B_R[0]$, then $x_a(T)\in B_R[0]$. In order to prove 
this, we consider the following set:
$$K := \Big\{t\in[0,T]:\norm{x_a(s)}\leq R \mbox{ for all 
}s\in[0,t]\Big\}.$$
It is clear that $K$ is nonempty because $0\in K$. We claim that 
$\sup K=T$. Suppose to the contrary that $t_0:=\sup K<T$. 
Then $\norm{x_a(t_0)}=R$. Hence,
$$\frac{d}{dt}\norm{x(t)}^2{\Big|}_{t=t_0}=2\p{f(t_0,x(t_0)),j}<0,$$
which implies that the function  $t\to \|x_a(t)\|$ is a decreasing function 
in the vicinity of the point $t_0$. Thus there exists $\varepsilon>0$ 
such that $\norm{x_a(t_0+\varepsilon)}\leq \norm{x_a(t_0)}=R$, 
which is a contradiction. Hence $\norm{x_a(T)}\leq R$ for all $a\in 
B_R[0]$, as claimed.
\par\medskip
Next, we define the functional $\g{\cdot,\cdot}:X\times X\to\R$ by 
$\g{x,y} := \max_{j\in J(y)}\p{x,j}$. Note that if $a\in\partial B_R[0]$, 
then for any $j\in J(a)$, we have
$$\p{F(a),j}=\p{-a+x_a(T),j}\leq-\norm{a}^2+\norm{x_a(T)}\,\norm{a}\leq0.$$
Thus the sign of $\g{F(a),a}$ is constant for all $a\in\partial B_R[0]$.
\par\medskip
In order to prove that the mapping $F$ is compact, it is enough to assume 
that the mapping $f:[0,T]\times X\to X$ is compact because in this case 
we have 
$$
F(a)=x_a(T)-a=\int_0^T f(s,x_a(s))\,ds\in T\,\overline{\co}\Big\{f([0,T]\times B_R[0])\Big\}$$
for all $a\in B_R[0]$, which implies that $F(B_R[0])$ is relatively compact.
\par\medskip
Therefore, using Theorem~\ref{tgb} we see that 
$0 \in \overline{F(B_R[0])}$. This means that there exists a sequence 
$\{a_n\}_{n\in\N}$ in $B_R[0]$ such that $F(a_n)\to 0$ as $n\to\infty$. Since
\begin{equation}\label{eq:dif.2}
F(a_n)=-a_n+\int_0^T f(s,x_{a_n}(s))\,ds
\end{equation}
for all $n\in\N$, and $\left\{\int_0^T f(s,x_{a_n}(s))\,ds\right\}_{n\in\N}$ 
is relatively compact, there exists a subsequence $\{a_{n_k}\}_{k\in\N}$ 
of $\{a_n\}_{n\in\N}$ such that
$$\int_0^T f(s,x_{a_{n_k}}(s))\,ds\to b\in B_R[0],$$
which implies, when combined with~\eqref{eq:dif.2}, that $a_{n_k}\to b$. 
Hence, using the continuity of $F$, we obtain that $F(a_{n_k})\to F(b)$. 
But, since $F(a_{n_k})\to 0$, we deduce that $F(b)=0$. 
This means that $x_b(0)=x_b(T)$, that is, there exists a periodic solution 
of differential equation~\eqref{eq:dif.1}.
\par\medskip
In other words, we have just proved the following result.

\begin{theorem}\label{thm:appl:infinite}
Let $X$ be a Banach space. Then the differential equation $x'(t)=f(t,x(t))$, 
where $t\in[0,T]$, has a periodic solution whenever $f:[0,T]\times 
X\to X$ is a completely continuous mapping which is Lipschitzian with respect 
to the second variable and there exists $R>0$ such that 
$\langle f(t,x),j(x)\rangle \le 0$ for all $t\in [0,T]$,  
all $\|x\|=R$ and for some $j(x)\in J(x)$.
\end{theorem}

\subsection{A final remark}

It may be worth noting that, instead of using Theorem~\ref{tgb},  
one can prove Theorem~\ref{thm:appl:finite}  
by appealing to  Brouwer's fixed point and that one 
can deduce Theorem~\ref{thm:appl:infinite} 
from Schauder's fixed point theorem (by proving that the  
Poincar\'e map, that is, $a \mapsto x_a(T)$, has a fixed point). 
One of the reasons for these facts is the following result, which provides  
an equivalence between these fixed point theorems and our theorem.

\begin{theorem}
The following theorems are equivalent:
\begin{enumerate}
\item[$\boldsymbol{(a)}$] Brouwer's fixed point theorem;
\item[$\boldsymbol{(b)}$] Schauder's fixed point theorem;
\item[$\boldsymbol{(c)}$] Theorem~\ref{tgb}.
\end{enumerate}
\end{theorem}
\begin{proof}

It is well known that $(a)\Rightarrow (b)$ holds; see, for instance, 
\cite[Chapter~2]{Sm}.  
The proof of Theorem~\ref{tgb} shows that $(b)\Rightarrow (c)$. 
Thus we only have to prove that $(c)\Rightarrow (a)$.

To this end, consider a nonempty, compact and convex subset $C$ of 
$\mathbb{R}^n,$ and let $f:C\to C$ be a continuous mapping. 
Since in $\mathbb{R}^n$ all norms are equivalent, without any loss of 
generality we may consider $\mathbb{R}^n$ endowed with the Euclidean norm 
$\|\cdot\|_2.$

Since the set $C$ is bounded, there exists a number $R>0$ such that 
$C\subseteq B_R[0]$. 
Define a mapping $F:B_R[0]\to \mathbb{R}^n$ by
$F(x):= x-f(P_C(x))$, where $P_C$ is the metric projection onto $C$.

Note that $F$ is a completely continuous mapping because $P_C$ 
is a continuous (even nonexpansive) mapping. 
Moreover, given $x\in S_R(0),$ 
we have $\langle F(x),x\rangle\geq R^2-\|f(P_C(x))\|_2 R\geq 0$, 
which means that if we define $\g{x,y}:=\langle x,y\rangle$, 
then $\g{F(x),x}$ has the same sign for all $x\in S_R(0)$. 
Under the above conditions, Theorem~\ref{tgb} yields a point  
$x_0\in B_R[0]$ such that $F(x_0)=0$. Hence $x_0=f(P_C(x_0))$, 
which means that $x_0\in C$ and therefore $P_C(x_0)=x_0$.  
We conclude that $x_0=f(x_0)$, as required. 
\end{proof}


\section*{Acknowledgements}

The first author was partially supported by
MTM~2015-65242-C2-1-P and by P08-FQM-03453. The second author was partially
supported by MTM 2015-65242-C2-2-P. The third author was partially
supported by the Israel Science Foundation (Grants no. 389/12 and 820/17),
by the Fund for the Promotion of Research at the Technion and by the
Technion General Research Fund. He also thanks Sergiu Hart for \cite{H}. 


\end{document}